\documentclass[A4, 11pt]{article}
\usepackage{amscd, amssymb, amsmath, amsthm}
\usepackage{latexsym}
\usepackage{upref}
\usepackage{epsfig}
\usepackage{mathrsfs}
\usepackage{float, floatflt}
\usepackage{indentfirst}
\usepackage{hyperref}
\hypersetup{
    colorlinks=true,
    linkcolor=blue,
    citecolor=blue,
    filecolor=blue,
    urlcolor=blue,
}
\usepackage{graphics,graphicx,color}
\usepackage{cite}
\usepackage{enumitem}
\usepackage{fancyhdr} 
\usepackage[T1]{fontenc}
\usepackage{mathptmx} 

\AtBeginDocument{
  \DeclareSymbolFont{AMSb}{U}{msb}{m}{n}
  \DeclareSymbolFontAlphabet{\mathbb}{AMSb}}  

\usepackage{rotating} 
\usepackage{booktabs} 
\usepackage{pb-diagram,pb-xy}       
\usepackage{tikz}                   
\usetikzlibrary{matrix,arrows}
\usepackage{multirow} 
\usepackage{color, colortbl} 
\definecolor{Gray}{gray}{0.93} 
\definecolor{LightCyan}{rgb}{0.88,1,1}

\theoremstyle{plain}
\newtheorem{theorem}{Theorem}[section]

\newtheorem{corollary}[theorem]{Corollary}
\theoremstyle{definition}
\newtheorem{definition}[theorem]{Definition}
\newtheorem{example}[theorem]{Example}
\theoremstyle{remark}

\numberwithin{equation}{section}

\makeatletter
\def\th@plain{%
  \thm@notefont{}
  \itshape 
}
\def\th@definition{%
  \thm@notefont{}
  \normalfont 
} \makeatother

\setlist{font=\normalfont}

\DeclareMathAlphabet{\cols}{OMS}{cmsy}{m}{n} %

\newcommand{\set}[1]{\{#1\}}
\newcommand{\cset}[2]{\set{{#1}\colon{#2}}}

\newcommand{\gyr}[2]{{\mathrm{gyr}[{#1}]}{#2}}
\newcommand{\qt}[1]{``#1''}
\newcommand{\R}{\mathbb{R}}
\newcommand{\norm}[1]{\|#1\|}
\newcommand{\gen}[1]{\langle#1\rangle}
\newcommand{\igyr}[2]{{\mathrm{gyr^{-1}}[{#1}]}{#2}}

\newcommand{\B}{\mathbb{B}}
\newcommand{\Bp}[1]{\left(#1\right)}
\newcommand{\vphi}{\varphi}
\newcommand{\D}{\mathbb{D}}
\newcommand{\abs}[1]{|#1|}
\newcommand{\Abs}[1]{\left|#1\right|}
\newcommand{\lbar}[1]{\overline{#1}}
\newcommand{\C}{\mathbb{C}}

\DeclareMathOperator{\aut}{Aut}

\renewcommand{\vec}[1]{\mathbf{#1}}

\begin{document}
\title{\bf{On metric structures of normed gyrogroups}}
\author{Teerapong Suksumran\,\href{https://orcid.org/0000-0002-1239-5586}{\includegraphics[scale=1]{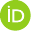}}\\
Center of Excellence in Mathematics and Applied Mathematics\\
Department of Mathematics\\
Faculty of Science, Chiang Mai University\\
Chiang Mai 50200, Thailand\\
{\tt teerapong.suksumran@cmu.ac.th}}
\date{}
\maketitle


\begin{abstract}
In this article, we indicate that the open unit ball in $n$-dimensional Euclidean space $\R^n$ admits norm-like functions compatible with the Poincar\'{e} and Beltrami--Klein metrics. This leads to the notion of a normed gyrogroup, similar to that of a normed group in the literature. We then examine topo-logical and geometric structures of normed gyrogroups. In particular, we prove that the normed gyrogroups are homogeneous and form left \mbox{invariant} metric spaces  and derive a version of the Mazur--Ulam theorem. We also give certain sufficient conditions, involving the right-gyrotranslation inequality and Klee's condition, for a normed gyrogroup to be a topological gyrogroup.
\end{abstract}
\thispagestyle{empty}

\section{Introduction}
Roughly speaking, a normed group is a group that comes with a compatible norm (also called a length function), similar to the case of normed linear spaces. A prominent example of a normed group is a finitely generated group with the word metric, which is one of the main ingredients in geometric group theory. The normed groups abound as an integral part in the theory of topological groups \cite{MR2743093}. They are blended objects that have significance in group theory, geometry, analysis, and topology, to name a few. 

In \cite{AU1988TRP}, Ungar studies a parametrization of the Lorentz transformation group. This leads to the formation of gyrogroup theory, a rich subject in mathematics \cite{TSKW2017MFE, TS2015GAG, TS2017IGU, AU2015PRL, MF2015HAM, MR3595192, TS2018CFE}. Loosely speaking, a gyrogroup is a group-like structure in which the associative law fails to satisfy. However, it obeys the {\it gyroassociative} law, a weak form of associativity, as well as the {\it loop property}, an algebraic rule \mbox{equivalent} to the {\it Bol identity} in loop theory.  One of the virtues of studying gyrogroups is an application in non-Euclidean geometry \cite{AU2008AHG}, where Ungar examines \mbox{analytic} \mbox{hyperbolic} geometry using the gyrolanguage.

Atiponrat \cite{MR3646419} and Cai et al. \cite{ZCSLWH2018NPL} pave the way for studying topological gyro-groups. In particular,  Atiponrat proves that in the class of topological gyrogroups, being a $T_0$-space is equivalent to being a $T_3$-space \cite[Theorem 3]{MR3646419}. Further, she attempts to extend the famous Birkhoff--Kakutani theorem by proving that every first-countable Hausdorff topological gyrogroup is premetrizable \cite[Theorem 4]{MR3646419}. The latter result is strengthened when Cai et al. prove that every first-countable Hausdorff topological gyrogroup is metrizable \cite[Theorem 2.3]{ZCSLWH2018NPL}. The achieved results inspire us to investigate topological properties of gyrogroups, which eventually bring us to the notion of a normed gyrogroup. This provides a large class of gyrogroups with a left invariant metric, and some of them are indeed topological gyrogroups. 

\section{Preliminaries}
Standard terminology and notation in algebra, topology, and geometry used throughout the article are defined as usual. In this section, we collect relevant \mbox{definitions} and elementary properties of gyrogroups for reference \cite{AU2008AHG, TS2016TAG}. The reader familiar with gyrogroup theory may skip this section. 

Let $G$ be a nonempty set equipped with a binary operation $\oplus$ on $G$. Denote by $\aut{G}$ the group of automorphisms of $(G, \oplus)$.

\begin{definition}[Gyrogroups]\label{def: gyrogroup}
A nonempty set $G$, together with a binary operation $\oplus$ on $G$, is called a {\it gyrogroup} if it satisfies the following axioms.
\begin{enumerate}
    \item[(G1)] There exists an element $e\in G$ such that $e\oplus a =
    a$ for all $a\in G$.
    \item[(G2)] For each $a\in G$, there exists an element $b\in G$ such that
$b\oplus a = e$.
    \item[(G3)] For all $a$, $b\in G$, there is an automorphism
$\gyr{a,b}{}\in\aut{G}$ such that
    \begin{equation}\tag{left gyroassociative law} a\oplus (b\oplus c) = (a\oplus b)\oplus\gyr{a,
    b}{c}\end{equation}
    for all $c\in G$.
    \item[(G4)] For all $a$, $b\in G$, $\gyr{a\oplus b, b}{} = \gyr{a,
    b}{}$.\hfill(left loop property)
\end{enumerate}
\end{definition}

\newpage

Note that the axioms in Definition \ref{def: gyrogroup} imply the right counterparts. In particular, any gyrogroup has a unique two-sided identity $e$, and an element $a$ of the gyrogroup has a unique two-sided inverse $\ominus a$. The automorphism $\gyr{a, b}{}$ is called the {\it gyroautomorphism} generated by $a$ and $b$. It is clear that every group satisfies the gyrogroup axioms (the gyroautomorphisms are the identity map) and hence is a gyrogroup. Conversely, any gyrogroup with {\it trivial} gyroautomorphisms forms a group. From this point of view, gyrogroups naturally generalize groups.

The table below summarizes some algebraic properties of gyrogroups \cite{AU2008AHG, TS2016TAG}, which will prove useful in studying topological and geometric aspects of gyro-groups in Sections \ref{sec: normed gyrogroup} and \ref{sec: topological structure}. We remark that gyroautomorphisms play an essential role in gyrogroup theory; for example, they appear as part of generic algebraic rules extended from group-theoretic identities.

\begin{table}[ht]
\centering
\begin{tabular}{ll}\hline
\rowcolor{LightCyan}{\hskip1.2cm}{\sc gyrogroup identity} & {\hskip0.8cm}{\sc name/reference}\\ \hline
$\ominus (\ominus a) = a$ & Involution of inversion\\ 
\rowcolor{Gray}$\ominus a\oplus(a\oplus x) = x$ & Left cancellation law\\
$\gyr{a, b}{c} = \ominus(a\oplus b)\oplus(a\oplus(b\oplus c))$ & Gyrator identity\\
\rowcolor{Gray}$\ominus (a\oplus b) = \gyr{a, b}{(\ominus b\ominus a)}$ & cf. $(ab)^{-1} = b^{-1}a^{-1}$\\
$(\ominus a\oplus b)\oplus\gyr{\ominus a, b}{(\ominus b\oplus c)} = \ominus a\oplus c$ & cf. $(a^{-1}b)(b^{-1}c) = a^{-1}c$\\
\rowcolor{Gray}$\gyr{\ominus a, \ominus b}{} = \gyr{a, b}{}$ & Even property\\
$\gyr{b, a} \,=\, \igyr{a, b}{}$, the inverse of $\gyr{a, b}{}$ & Inversive symmetry\\
\rowcolor{Gray}$\vphi(\gyr{a, b}{c}) = \gyr{\vphi(a), \vphi(b)}{\vphi(c)}$ & Gyration preserving under\\
\rowcolor{Gray}{} & a gyrogroup homomorphism $\vphi$\\
$L_a\circ L_b = L_{a\oplus b}\circ\gyr{a, b}{}$ & Composition law for\\
{} & left gyrotranslations\\
\hline
\end{tabular}
\caption{Algebraic properties of gyrogroups (cf. \cite{AU2008AHG, TS2016TAG}).}\label{tab: properties of gyrogroups}
\end{table}

\section{Normed gyrogroups and concrete examples}\label{sec: normed gyrogroup}
In this section, we establish that any gyrogroup with an appropriate {\it length function}, called a {\it gyronorm}, has the metric structure and hence is a Hausdorff space. We also exhibit a few examples of well-known gyrogroups that have gyro-norms.

\subsection{The definition and basic properties}
\begin{definition}[Gyronorms]\label{def: length function on gyrogroup} Let $G$ be a gyrogroup. A function $\norm{\cdot}\colon G\to \R$ is called a {\it gyronorm} on $G$ if the following properties hold:
\begin{enumerate}
\item\label{item: positivity} $\norm{x}\geq 0$  for all $x\in G$ and $\norm{x} = 0$ if and only if $x = e$;\hfill(positivity)
\item\label{item: invariant under taking inverses} $\norm{\ominus x} = \norm{x}$ for all $x\in G$;\hfill(invariant under taking inverses)
\item\label{item: subadditivity} $\norm{x\oplus y}\leq \norm{x}+\norm{y}$ for all $x, y\in G$;\hfill(subadditivity)
\item\label{item: invariant under gyrations}$\norm{\gyr{a, b}{x}} = \norm{x}$ for all $a, b, x\in G$.\hfill(invariant under gyrations)
\end{enumerate}
\end{definition}
Any gyrogroup with a gyronorm is called a {\it normed} gyrogroup. We remark that the term \qt{gyronorm} is quite different from what Ungar used in Chapter 4 of \cite{AU2008AHG}.  In Section \ref{sec: examples}, we give several concrete examples of gyrogroups with a gyronorm. Clearly, Definition \ref{def: length function on gyrogroup} is a generalization of the notion of a {\it group}-norm \cite[p. 8]{MR2743093}, which in turn is motivated by norms on linear spaces. Furthermore, any gyrogroup may be viewed as a normed gyrogroup with a gyronorm defined by
$$
\norm{x} = 
\begin{cases}
0 & \textrm{if }x = e;\\
1 & \textrm{if }x \ne e.\\
\end{cases}
$$

\begin{theorem}\label{thm: metric induced by length function}
Let $G$ be a normed gyrogroup. Define
\begin{equation}
d(x, y) = \norm{\ominus x\oplus y}
\end{equation}
for all $x, y\in G$. Then $d$ is a metric on $G$ and so $(G, d)$ forms a metric space.
\end{theorem}
\begin{proof}
By definition, $d(x, y) = \norm{\ominus x\oplus y}\geq 0$ for all $x, y\in G$. Clearly, $d(x, x)  = \norm{e} = 0$ for all $x\in G$. Suppose that  $d(x, y) = 0$. Then $\norm{\ominus x\oplus y} = 0$. By definition, $\ominus x\oplus y = e$. Hence, $x = y$ by the left cancellation law.

Let $x, y, z\in G$. Using appropriate properties of gyrogroups in Table \ref{tab: properties of gyrogroups}, together with the defining properties of a gyronorm, we obtain 
$$
d(y, x) = \norm{\ominus y\oplus x} =
 \norm{\ominus(\ominus y\oplus x)}
= \norm{\gyr{\ominus y, x}{(\ominus x\oplus y)}}
= \norm{\ominus x\oplus y}
=d(x, y).
$$
Furthermore, we obtain
\begin{align*}
d(x, z) &= \norm{\ominus x\oplus z}\\
{} &= \norm{(\ominus x\oplus y)\oplus \gyr{\ominus x, y}{(\ominus y\oplus z)}}\\
{} &\leq \norm{\ominus x\oplus y}+\norm{\gyr{\ominus x, y}{(\ominus y\oplus z)}}\\
{} &= \norm{\ominus x\oplus y}+ \norm{\ominus y\oplus z}\\
{} &= d(x, y) + d(y, z).
\end{align*}
This proves that $d$ satisfies the defining properties of a metric.
\end{proof}

The metric $d$ induced by a gyronorm on $G$ in Theorem \ref{thm: metric induced by length function} is called a {\it gyronorm} metric. Whenever we say that $G$ is a normed gyrogroup, we assume that $G$ is endowed with the corresponding gyronorm metric and that $G$ carries the topology induced by this metric, unless mentioned otherwise. It is clear that every isometry of a normed gyrogroup to itself is a homeomorphism for the inverse of an isometry is again an isometry. Next, we show that known gyrogroups in the literature possess gyronorms.

\subsection{Concrete examples}\label{sec: examples}
\subsubsection{An $n$-dimensional Euclidean version of the Einstein gyrogroup}
Let $\B$ denote the open unit ball in $n$-dimensional Euclidean space $\R^n$, that is,
\begin{equation}
\B = \cset{\vec{v}\in\R^n}{\norm{\vec{v}}<1},
\end{equation}
where $\norm{\cdot}$ denotes the {\it Euclidean} norm on $\R^n$. The {\it Einstein} gyrogroup consists of $\B$, together with Einstein addition $\oplus_E$ given by
\begin{equation}\label{eqn: Einstein addition}
\vec{u}\oplus_E\vec{v} =
\dfrac{1}{1+\gen{\vec{u},\vec{v}}}\Bp{\vec{u}+
\dfrac{1}{\gamma_{\vec{u}}}\vec{v} +
\dfrac{\gamma_{\vec{u}}}{1+\gamma_{\vec{u}}}\gen{\vec{u},\vec{v}}\vec{u}}
\end{equation}
for all $\vec{u}, \vec{v}\in\B$, where $\gamma_{\vec{u}}$ is the {\it Lorentz factor} given by 
$\gamma_{\vec{u}} = \dfrac{1}{\sqrt{1-\norm{\vec{u}}^2}}.
$
The zero vector $\vec{0}$ acts an the identity of $\B$ under $\oplus_E$. For each $\vec{v}\in\B$, the negative vector $-\vec{v}$ acts as the inverse of $\vec{v}$ with respect to Einstein addition. By Proposition 2.4 of \cite{SKJL2013UBL}, the gyroautomorphisms of $(\B, \oplus_E)$ are orthogonal in the sense that
$$
\norm{\gyr{\vec{u}, \vec{v}}{\vec{w}}} = \norm{\vec{w}}
$$
for all $\vec{u}, \vec{v}, \vec{w}\in\B$.

Define a function $\norm{\cdot}_E$ on $\B$ by the equation
\begin{equation}\label{eqn: length function on Einstein gyrogroup}
\norm{\vec{v}}_E = \tanh^{-1}{\norm{\vec{v}}},\qquad \vec{v}\in\B,
\end{equation}
where $\tanh^{-1}$ denotes the inverse of the hyperbolic tangent function on $\R$.

\begin{theorem}\label{thm: length function on Einstein gyrogroup}
The function $\norm{\cdot}_E$ defined by \eqref{eqn: length function on Einstein gyrogroup} is a gyronorm on the Einstein gyrogroup.
\end{theorem}
\begin{proof}
Since $\tanh{r}\geq 0$ if and only $r\geq 0$, it follows that $\norm{\vec{v}}_E\geq 0$ for all $\vec{v}\in\B$. Note that $\norm{\vec{0}}_E = \tanh^{-1}0 = 0$. Suppose that $\norm{\vec{v}}_E = 0$. Then $0 = \tanh^{-1}{\norm{\vec{v}}}$, which implies $\norm{\vec{v}} = 0$. Hence, $\vec{v} = \vec{0}$. Let $\vec{v}\in\B$. Then $$\norm{\ominus \vec{v}} _E= \norm{-\vec{v}}_E = \tanh^{-1}{\norm{-\vec{v}}} = \tanh^{-1}{\norm{\vec{v}}} = \norm{\vec{v}}_E.$$

Let $\vec{u}, \vec{v}\in\B$. Applying Proposition 3.3 of \cite{SKJL2013UBL} and Lemma 3.2 (iv) of \cite{SKJL2013UBL} gives $$\norm{\vec{u}\oplus_E \vec{v}}_E \leq \norm{\vec{u}}_E+\norm{\vec{v}}_E.$$ Let $\vec{u}, \vec{v}, \vec{w}\in\B$. Then 
\begin{equation*}
\norm{\gyr{\vec{u}, \vec{v}}{\vec{w}}}_E = \tanh^{-1}{\norm{\gyr{\vec{u}, \vec{v}}{\vec{w}}}} = \tanh^{-1}{\norm{\vec{w}}} = \norm{\vec{w}}_E.\qedhere
\end{equation*}
\end{proof}

It follows from Theorems \ref{thm: metric induced by length function} and \ref{thm: length function on Einstein gyrogroup} that
\begin{equation}
d_E(\vec{u}, \vec{v}) = \tanh^{-1}{\norm{-\vec{u}\oplus_E \vec{v}}}
\end{equation}
defines a metric on $\B$. This metric is called the {\it rapidity} metric on the Einstein gyrogroup \cite{SKJL2013UBL, AU2008AHG}. It is known that  the rapidity metric on the Einstein gyrogroup agrees with the {\it Cayley--Klein} metric on the Beltrami--Klein model of $n$-dimensional hyperbolic geometry \cite[p. 1233]{SKJL2013UBL}.

Note that the Euclidean norm is indeed a gyronorm on the Einstein gyrogroup. This follows from the fact that
$$
\norm{\vec{u}\oplus_E\vec{v}} \leq \norm{\vec{u}}\oplus \norm{\vec{v}} = \dfrac{\norm{\vec{u}} + \norm{\vec{v}}}{1+\norm{\vec{u}}\norm{\vec{v}}}\leq \norm{\vec{u}} + \norm{\vec{v}},
$$
where the first inequality is worked out in Proposition 3.3 of \cite{SKJL2013UBL} and $\oplus$ is the restricted Einstein addition on the open interval $(-1, 1)$ given by
$
r\oplus s = \dfrac{r+s}{1+rs}
$
for all $r, s\in (-1, 1)$. Denote by $d_e$ the gyronorm metric induced by the Euclidean norm. That is,
\begin{equation}
d_e(\vec{u},\vec{v}) = \norm{-\vec{u}\oplus_E\vec{v}}
\end{equation}
for all $\vec{u},\vec{v}\in\B$. In fact, $d_e$ is known as the {\it Einstein gyrometric} \cite[p. 222]{AU2008AHG}. Our results provide an elegant proof that the Einstein gyrometric is indeed a metric on the open unit ball of $\R^n$.

Note that $d_e(\vec{u}, \vec{v}) \leq d_E(\vec{u}, \vec{v})$
for all $\vec{u}, \vec{v}\in\B$ because $x\mapsto x - \tanh^{-1}{x}$ defines a strictly decreasing function on the open interval $(0, 1)$. This implies that the topology generated by $d_E$ is finer than the topology generated by $d_e$. Next, we prove that the topology generated by $d_e$ is finer than the topology generated by $d_E$. Let $\vec{u}\in\B$ and let $\epsilon>0$. Choose $\delta = \tanh{\epsilon}$. Let $\vec{v}\in B_{d_e}(\vec{u}, \delta)$. Then $d_e(\vec{u}, \vec{v})<\delta$, that is,
$\norm{-\vec{u}\oplus_E\vec{v}} < \tanh{\epsilon}$. It follows that
$$
d_E(\vec{u},\vec{v}) = \tanh^{-1}{\norm{-\vec{u}\oplus_E\vec{v}}} < \epsilon
$$
for $\tanh^{-1}$ is a strictly increasing function on its domain. Thus, $\vec{v}\in B_{d_E}(\vec{u}, \epsilon)$. This proves that $B_{d_e}(\vec{u}, \delta)\subseteq B_{d_E}(\vec{u}, \epsilon)$. Therefore, $d_e$ and $d_E$ generate the same topology on $\B$.

\subsubsection{An $n$-dimensional Euclidean version of the M\"{o}bius gyrogroup}
The {\it M\"{o}bius} gyrogroup consists of the same underlying set as the Einstein \mbox{gyrogroup}, but its binary operation, called {\it M\"{o}bius addition}, is defined by
\begin{equation}\label{eqn: Mobius addition}
\vec{u}\oplus_M\vec{v} = \dfrac{(1 + 2\gen{\vec{u},\vec{v}} +
\norm{\vec{v}}^2)\vec{u} + (1 - \norm{\vec{u}}^2)\vec{v}}{1 +
2\gen{\vec{u},\vec{v}} + \norm{\vec{u}}^2\norm{\vec{v}}^2}
\end{equation}
for all $\vec{u}, \vec{v}\in\B$. The zero vector acts an the identity of $\B$ under $\oplus_M$. For each $\vec{v}\in\B$, the negative vector $-\vec{v}$ acts as the inverse of $\vec{v}$ with respect to M\"{o}bius addition. By Proposition 2.4 of \cite{SKJL2013UBL}, the gyroautomorphisms of $(\B, \oplus_M)$ are orthogonal in the sense that
$$
\norm{\gyr{\vec{u}, \vec{v}}{\vec{w}}} = \norm{\vec{w}}
$$
for all $\vec{u}, \vec{v}, \vec{w}\in\B$.

By Proposition 2.3 of \cite{SKJL2013UBL}, the map $\Phi$ defined by
\begin{equation}
\Phi(\vec{v}) = \dfrac{2}{1+\norm{\vec{v}}^2}\vec{v},\qquad \vec{v}\in\B
\end{equation}
is a gyrogroup isomorphism from $(\B, \oplus_M)$ to $(\B, \oplus_E)$. In particular, $\Phi$ is a bijection from $\B$ to itself. Let $\norm{\cdot}_E$ be the gyronorm on the Einstein gyrogroup defined by \eqref{eqn: length function on Einstein gyrogroup}. Define a function $\norm{\cdot}_M$ by the equation
\begin{equation}\label{eqn: length function of Mobius gyrogroup}
\norm{\vec{v}}_M = \dfrac{1}{2}\norm{\Phi(\vec{v})}_E,\qquad \vec{v}\in\B.
\end{equation}

\begin{theorem}\label{thm: length function on Mobius gyrogroup}
The function $\norm{\cdot}_M$ defined by \eqref{eqn: length function of Mobius gyrogroup} is a gyronorm on the M\"{o}bius gyrogroup.
\end{theorem}
\begin{proof}
The theorem follows directly from the fact that $\norm{\cdot}_E$ defines a gyronorm on $(\B, \oplus_E)$ and that $\Phi$ is a gyrogroup isomorphism.
\end{proof}

By Theorems \ref{thm: metric induced by length function} and \ref{thm: length function on Mobius gyrogroup}, 
\begin{equation}
d_M(\vec{u}, \vec{v}) = \dfrac{1}{2}\tanh^{-1}{\norm{\Phi(-\vec{u}\oplus_M \vec{v})}}
\end{equation}
defines a metric on $\B$. This metric is called the {\it rapidity} metric on the M\"{o}bius \mbox{gyrogroup} \cite{SKJL2013UBL, AU2008AHG}. It is known that the rapidity metric on the M\"{o}bius gyrogroup is half of the {\it Poincar\'{e}} metric with curvature $-1$ on the Poincar\'{e} model of $n$-dimensional hyperbolic geometry \cite[Theorem 3.7]{SKJL2013UBL}.

\section{Topological and geometric structures}\label{sec: topological structure}
Normed gyrogroups have nice topological and geometric structures. Further, they share certain remarkable analogies with normed groups. In fact, several of the results proved in this section are inspired by the expository article of Bingham and Ostaszewski \cite{MR2743093} that treats normed and topological groups. Roughly speaking, normed gyrogroups are left invariant, homogeneous, and isotropic.

\subsection{Topological and geometric properties}
\begin{theorem}\label{thm: left gyrotranslation invariant under length metric}
Let $G$ be a normed gyrogroup. Then the gyronorm metric is in-variant under left gyrotranslation:
\begin{equation}
d(a\oplus x, a\oplus y) = d(x, y)
\end{equation}
for all $a, x, y\in G$. Hence, every left gyrotranslation of $G$ is an isometry of $G$ with respect to the gyronorm metric.
\end{theorem}
\begin{proof}
Let $a\in G$. Recall that the {\it left gyrotranslation} by $a$, denoted by $L_a$, is defined by $L_a(x) = a\oplus x$ for all $x\in G$. By Theorem 18 (1) of \cite{TS2016TAG}, $L_a$ is a bijection from $G$ to itself. 

Next, we prove that the gyronorm metric $d$ is invariant under $L_a$. Let $x, y\in G$. Using appropriate properties of gyrogroups in Table \ref{tab: properties of gyrogroups}, together with the defining properties of a gyronorm, we obtain 
\begin{align*}
d(L_a(x), L_a(y)) &= \norm{\ominus (a\oplus x)\oplus (a\oplus y)}\\
{} &= \norm{\gyr{a, x}{(\ominus x\ominus a)}\oplus(a\oplus y)}\\
{} &= \norm{(\ominus x\ominus a)\oplus\gyr{x, a}{(a\oplus y)}}\\
{} &= \norm{(\ominus x\ominus a)\oplus\gyr{\ominus x, \ominus a}{(a\oplus y)}}\\
{} &= \norm{\ominus x\oplus y}\\
{} &= d(x, y).\qedhere
\end{align*}
\end{proof}

\begin{corollary}\label{cor: isometry imply homeomorphism}
If $G$ is a normed gyrogroup, then every left gyrotranslation of $G$ is a homeomorphism.
\end{corollary}

\begin{theorem}\label{thm: automorphism invariant under length metric}
Let $G$ be a normed gyrogroup. If $\tau\in\aut{G}$ and $\norm{\tau(x)} = \norm{x}$ for all $x\in G$, then $\tau$ is an isometry of $G$ with respect to the gyronorm metric.
\end{theorem}
\begin{proof}
By assumption, 
$$
d(\tau(x), \tau(y)) = \norm{\ominus \tau(x)\oplus\tau(y)} = \norm{\tau(\ominus x\oplus y)} = \norm{\ominus x\oplus y} = d(x, y)
$$
and so $\tau$ defines an isometry of $G$.
\end{proof}

\begin{corollary}\label{cor: gyr[a, b] as isometry}
If $G$ is a normed gyrogroup, then the gyroautomorphisms of $G$ are isometries (and also homeomorphisms) of $G$.
\end{corollary}

\begin{theorem}[Homogeneity]
If $G$ is a normed gyrogroup, then $G$ is homogeneous in the sense that if $x$ and $y$ are arbitrary points of $G$, then there is an isometry $T\colon G\to G$ (and also a homeomorphism of $G$) such that $T(x) = y$.
\end{theorem}
\begin{proof}
Let $x, y\in G$. Define $T = L_y\circ L_{\ominus x}$. By Theorem \ref{thm: left gyrotranslation invariant under length metric}, $T$ is an isometry of $G$. Further, 
$$
T(x) = (L_y\circ L_{\ominus x})(x) = L_y(L_{\ominus x}(x)) = L_y(\ominus x\oplus x) = L_y(e) = y\oplus e = y.
$$
Thus, $G$ is homogeneous.
\end{proof}

\begin{theorem}[Isotropy]\label{thm: isotropic geometry, gyrogroup with length}
If $G$ is a nondegenerate normed gyrogroup; that is, $G$ has a nonidentity gyroautomorphism, then $G$ is isotropic in the sense that for each point $p\in G$, there exists a nonidentity isometry $T$ of $G$ such that $T(p) = p$.
\end{theorem}
\begin{proof}
Let $p$ be an arbitrary point of $G$. Let $\tau$ be a {\it nonidentity} gyroautomorphism of $G$. Then $\tau(e) = e$. Define $T = L_p\circ \tau\circ L_{\ominus p}$. Note that $T$ is an isometry of $G$, being the composite of isometries of $G$. Further, $T(p) = p$. Note that $T$ is not the identity transformation of $G$; otherwise, we would have $I = L_p\circ\tau\circ L_{\ominus p} = L_p\circ\tau\circ L_p^{-1}$ and would have $\tau = I$, a contradiction.
\end{proof}

Recall that the famous Mazur--Ulam theorem states that any isometry between normed linear spaces over $\R$ that fixes the zero vector must be linear; see, for instance, \cite[Theorem 1.3.5]{MR1957004}. Extensions of the Mazur--Ulam theorem are studied by Rassias \cite{MR2316585, MR1972385} and by Rassias et al. \cite{MR1111437, MR1845699}. Further, the Mazur--Ulam theorem is examined in the setting of {\it gyrovector spaces} by Abe \cite{TA2014GPM} and by Abe and Hatori \cite{MR3412000}. Here, we prove a normed-gyrogroup version of the Mazur--Ulam theorem.

\begin{theorem}
Let $G$ be a normed gyrogroup. If $f$ is an isometry of $G$ with respect to the gyronorm metric, then 
\begin{equation}
f = L_{f(e)}\circ \rho,
\end{equation}
where $\rho$ is an isometry of $G$ that leaves the gyrogroup identity fixed.
\end{theorem}
\begin{proof}
Suppose that $f$ is an isometry of $G$. By definition, $f$ is a permutation of $G$. By Proposition 19 of \cite{TS2016TAG}, $f = L_{f(e)}\circ\rho$, where $\rho$ is a permutation of $G$ that fixes $e$. As in the proof of Theorem 18 of \cite{TS2016TAG}, $ L_{f(e)}^{-1} = L_{\ominus f(e)}$ and so $\rho = L_{\ominus f(e)}\circ f$. Hence, $\rho$ is an isometry of $G$, being the composite of isometries of $G$.
\end{proof}

\subsection{A characterization of normed gyrogroups}

Note that the gyronorm of an arbitrary normed gyrogroup can be recovered by its corresponding metric:
\begin{equation}
\norm{x} = d(e, x),\qquad x\in G.
\end{equation}
It turns out that Theorem \ref{thm: left gyrotranslation invariant under length metric} provides a characterizing property of normed gyro-groups, as shown in the following theorem.

\begin{theorem}\label{thm: length function from metric invariant under left gyrotranslation}
Let $G$ be a gyrogroup with a metric $d$. If $d$ is invariant under left gyrotranslation, that is,
$$
d(a\oplus x, a\oplus y) = d(x, y)
$$
for all $a, x, y\in G$, then $\norm{x} = d(e, x)$ defines a gyronorm on $G$ that generates the same metric.
\end{theorem} 
\begin{proof}
It is clear that $\norm{x} \geq 0$ and $\norm{x} = 0$ if and only if $x = e$. Let $x\in G$. By the left gyrotranslation invariant, $\norm{\ominus x} = d(e, \ominus x) = d(x\oplus e, x\ominus x) = d(x, e) = \norm{x}$. 

Let $x, y\in G$. Direct computation shows that 
$$
\norm{x\oplus y} = d(x\oplus y, e) = d(y, \ominus x) \leq d(y, e)+d(e, \ominus x) = \norm{x}+\norm{y}.
$$
Let $a, b, x\in G$. By the gyrator identity, 
\begin{align*}
\norm{\gyr{a, b}{x}} &= d(\gyr{a, b}{x}, e) \\
{} &= d(\ominus(a\oplus b)\oplus(a\oplus(b\oplus x)), e)\\
{} &= d(a\oplus(b\oplus x), a\oplus b))\\
{} &= d(b\oplus x, b)\\
{} &= d(x, e)\\
{} &= \norm{x}.\qedhere
\end{align*}
\end{proof}

In view of Theorems \ref{thm: metric induced by length function}, \ref{thm: left gyrotranslation invariant under length metric}, and \ref{thm: length function from metric invariant under left gyrotranslation}, there is a one-to-one correspondence between the class of normed gyrogroups and the class of gyrogroups with a left-gyrotranslation-invariant metric:
$$
\set{\textrm{Normed gyrogroups}} \quad\longleftrightarrow\quad \set{(G, d), d \textrm{ left-gyrotranslation-invariant metric}}.
$$
One of the advantages of Theorem \ref{thm: length function from metric invariant under left gyrotranslation} is illustrated in the example below.
 
\begin{example}[The complex M\"{o}bius gyrogroup]
The Poincar\'{e} disk model consists of the open unit disk in the complex plane,
\begin{equation}
\D = \cset{z\in\C}{\abs{z}<1},
\end{equation}
and the (complex version of) Poincar\'{e} metric defined by
\begin{equation}\label{eqn: complex Mobius addition}
d_P(w, z) = 2\tanh^{-1}\Abs{\dfrac{w-z}{1-\lbar{w}z}}
\end{equation}
for all $w, z\in\D$. Here, the factor $2$ is added to \eqref{eqn: complex Mobius addition} so that the metric corresponds to a curvature of $-1$. 

A {\it complex} version of M\"{o}bius addition is defined by
\begin{equation}\label{eqn: complex mobius addition}
a\oplus_M b = \dfrac{a+b}{1+\lbar{a}b},\qquad a, b\in\D,
\end{equation}
which gives $\D$ the gyrogroup structure \cite{AU2008FMG}. It is not difficult to check that $0$ is the identity of $\D$, that the inverse of $a$ is $-a$, and that the gyroautomorphism generated by $a$ and $b$ is a disk rotation corresponding to the unimodular complex $\dfrac{1+a\lbar{b}}{1+\lbar{a}b}$. Using \eqref{eqn: complex Mobius addition} and \eqref{eqn: complex mobius addition}, we have by inspection that
$$
d_P(a\oplus_M w, a\oplus_M z) = d_P(w, z)
$$
for all $a, w, z\in\D$. Hence, by Theorem \ref{thm: length function from metric invariant under left gyrotranslation}, $\D$ forms a normed gyrogroup whose gyronorm is given by
$\norm{z} = 2\tanh^{-1}{\abs{z}}$ for all $z\in\D$. This leads to the well-known fact  that any M\"{o}bius transformation (also called a {\it conformal} self-map) of $\D$ of the form $$z\mapsto \dfrac{a+z}{1+\lbar{a}z},\qquad z\in\D,$$
where $a$ is a fixed element in $\D$, is an isometry of $\D$ with respect to the Poincar\'{e} metric. The study of M\"{o}bius transformations is an important topic in \mbox{mathematics}. This is evidenced by characterizations of M\"{o}bius transformations found in the \mbox{literature}; see, for instance, \cite{MR1371271, MR1485479, MR1653398, DEMIREL2013457, MR2729244, MR1772592, MR1850963}.
\hfill\qedsymbol\end{example}

\subsection{Sufficient conditions to be a topological gyrogroup}
In \cite{MR0047250}, Klee shows that in the class of groups with a metric $d$ the condition that $d(xy, ab)\leq d(x, a)+d(y, b)$ is equivalent to the bi-invariant of $d$. This motivates the following theorem for normed gyrogroups:

\begin{theorem}\label{thm: right gyro ineq iff Klee condition}
Let $G$ be a normed gyrogroup with the corresponding metric $d$. Then the following conditions are equivalent:
\begin{enumerate}[label=(\Roman*)]
\item\label{item: right-gyrotranslation ineq} Right-gyrotranslation inequality: $d(x\oplus a, y\oplus a) \leq d(x, y)$ for all $a, x, y\in G$;
\item\label{item: Klee's condition} Klee's condition:  $d(x\oplus y, a\oplus b)\leq d(x, a)+d(y, b)$ for all $a, b, x, y\in G$.
\end{enumerate}
\end{theorem}
\begin{proof} Assume that the right-gyrotranslation inequality holds. Then we have
\begin{align*}
d(x\oplus y, a\oplus b) &\leq d(x\oplus y, x\oplus b) + d(x\oplus b, a\oplus b)\\
{} &= d(y, b) + d(x\oplus b, a\oplus b)\\
{} &\leq d(y, b) + d(x, a)\\
{} &= d(x, a)+d(y, b)
\end{align*}
for all $a, b, x, y\in G$. Conversely, if Klee's condition holds, then
\begin{equation*}
d(x\oplus a, y\oplus a) \leq d(x, y)+d(a, a) = d(x, y)
\end{equation*}
for all $a, x, y\in G$.
\end{proof}

Recall that a gyrogroup $G$ endowed with a topology is called a {\it topological} gyrogroup if (\hyperlink{i}{i}) the gyroaddition map $(x, y)\mapsto x\oplus y$ is jointly continuous and (\hyperlink{ii}{ii}) the inversion map $x\mapsto \ominus x$ is continuous \cite[Definition 1]{MR3646419}. In general, a normed gyrogroup need not be a topological gyrogroup. The conditions mentioned in \mbox{Theorem} \ref{thm: right gyro ineq iff Klee condition} are sufficient conditions for a normed gyrogroup to be a \mbox{topological} gyrogroup. It is still an open question whether these conditions are necessary. 

\begin{theorem}
Let $G$ be a normed gyrogroup. If one of the conditions in Theorem \ref{thm: right gyro ineq iff Klee condition} holds, then $G$ is a topological gyrogroup with respect to the topology induced by the gyronorm metric.
\end{theorem}
\begin{proof}
Denote by $A$ the gyroaddition map: $A(x, y) = x\oplus y$. Let $(x, y)$ be an arbitrary point of $G\times G$ and let $V$ be a neighborhood of $A(x, y) = x\oplus y$. By definition, there is an $\epsilon>0$ such that $B(x\oplus y, \epsilon)\subseteq V$. Define $S = B(x, \epsilon/2)$ and $T = B(y, \epsilon/2)$. Set $U = S\times T$. Since $S$ and $T$ are open in $G$, it follows that $U$ is a neighborhood of $(x, y)$ in $G\times G$. Let $(a, b)\in U$. Then $a\in S$ and $b\in T$. By assumption,
$$
d(x\oplus y, a\oplus b) \leq d(x, a)+d(y, b) < \epsilon.
$$
Hence, $A(a, b) = a\oplus b\in B(x\oplus y, \epsilon)\subseteq V$ and so $A(U)\subseteq V$. This proves that $A$ is continuous.

Denote by $\iota$ the inversion map of $G$. By the right-gyrotranslation inequality, 
\begin{align*}
d(\iota(x), \iota(y)) &= d(\ominus x, \ominus y)\\
{} &\leq d(e, \ominus y\oplus x)\\
{} &= d(y\oplus e, y\oplus(\ominus y\oplus x))\\
{} &= d(x, y)
\end{align*}
for all $x, y\in G$. This also implies that $d(x, y)\leq d(\iota(x), \iota(y))$ for all $x, y\in G$ because $\iota = \iota^{-1}$. Thus, $d(\iota(x), \iota(y)) = d(x, y)$ for all $x, y\in G$ and so $\iota$ is an isometry of $G$. Hence, $\iota$ is continuous. This proves that $G$ is a topological gyrogroup.
\end{proof}

According to Theorem 2.18 of \cite{MR2743093}, the group-norm of a group $\Gamma$ is {\it abelian}, that is, $$\norm{gh} = \norm{hg}$$ for all $g, h\in\Gamma$ if and only if the metric induced by this group-norm is bi-invariant. This motivates the following theorem for normed gyrogroups:

\begin{theorem}
Let $G$ be a normed gyrogroup with the corresponding metric $d$. Then the following conditions are equivalent.
\begin{enumerate}[label=(\Roman*)]
\item\label{item: commutative-like} Commutative-like condition: $\norm{(a\oplus x)\oplus\gyr{a, x}{(y\ominus a)}} = \norm{x\oplus y}$ for all $a, x, y\in G$.
\item\label{item: bi-gyrotranslation invariant} Bi-gyrotranslation invariant: $d(x\oplus a, y\oplus a) = d(x, y) = d(a\oplus x, a\oplus y)$ for all $a, x, y\in G$.
\end{enumerate}
\end{theorem}
\begin{proof}
 Let $a, x, y\in G$. Direct computation shows that
\begin{align}\label{eqn: in proof: right gyrotranslation invariant}
\begin{split}
d(x\oplus a, y\oplus a) &= \norm{\ominus (x\oplus a)\oplus(y\oplus a)}\\
{} &= \norm{\gyr{x, a}{(\ominus a\ominus x)}\oplus (y\oplus a)}\\
{} &= \norm{(\ominus a\ominus x)\oplus\gyr{a, x}{(y\oplus a)}}\\
{} &= \norm{(\ominus a\ominus x)\oplus\gyr{\ominus a, \ominus x}{(y\ominus (\ominus a))}}\\
{} &= \norm{\ominus x\oplus y}\\
{} &= d(x, y).
\end{split}
\end{align}
Hence, $d$ is invariant under right gyrotranslation. By Theorem \ref{thm: left gyrotranslation invariant under length metric}, $d$ is invariant under left gyrotranslation as well. Conversely, computation as in \eqref{eqn: in proof: right gyrotranslation invariant} with $\ominus a$ in place of $a$ and $\ominus x$ in place of $x$ gives
\begin{equation*}
\norm{x\oplus y} = d(\ominus x, y) = d(\ominus x\ominus a, y\ominus a) = \norm{(a\oplus x)\oplus\gyr{a, x}{(y\ominus a)}}.\qedhere
\end{equation*}
\end{proof}

\vspace{0.3cm}
\noindent{\bf Acknowledgements.} The author would like to thank Themistocles M. Rassias and Watchareepan Atiponrat for their collaboration.

\bibliographystyle{amsplain}\addcontentsline{toc}{section}{References}
\bibliography{References}
\end{document}